\documentclass[12pt]{amsart}

\usepackage[english]{babel}
\usepackage{amsmath,amssymb,amsfonts,amsthm,amsopn,amstext,amsxtra,amscd}
\usepackage{bm,mathrsfs,mathtools}
\usepackage{cite}
\usepackage{graphicx}
\usepackage{float}
\usepackage{xcolor}
\usepackage{url}
\usepackage{todonotes}

\usepackage[colorlinks,linkcolor=blue,anchorcolor=blue,citecolor=blue,backref=page]{hyperref}
\hypersetup{breaklinks=true}

\theoremstyle{plain}
\newtheorem{theorem}{Theorem}[section]
\newtheorem{lemma}[theorem]{Lemma}

\newtheorem{proposition}[theorem]{Proposition}

\theoremstyle{definition}

\theoremstyle{remark}

\numberwithin{equation}{section}
\numberwithin{table}{section}

\DeclareMathOperator{\Kl}{Kl}

\newcommand{\eps}{\varepsilon}
\newcommand{\om}{\omega}
\newcommand{\wt}{\widehat}

\newcommand{\e}{\mathbf{e}}
\newcommand{\calB}{\mathcal{B}}
\newcommand{\calP}{\mathcal{P}}
\newcommand{\abs}[1]{\left|#1\right|}

\title[Kloosterman sign changes with $P_5$ moduli]
{Kloosterman sign changes with moduli having at most five prime factors}

\author{Yixiu Xiao}
\address{School of Mathematical Sciences, Shanghai Jiao Tong University, 800 Dongchuan Road, Shanghai 200240, China}
\email{yixiuxiao98@gmail.com}

\subjclass[2020]{Primary 11L05; Secondary 11N36}
\keywords{Kloosterman sums, sign changes, almost-prime moduli, Selberg sieve, truncated divisor functions}

\begin{document}

\begin{abstract}
On square-free moduli $q\in(X,2X]$ having at most five prime factors, we prove
that each sign of the normalized Kloosterman sum $\Kl(1;q)$ occurs
$\gg X/\log X$ times.  This improves the recent unconditional result of Zhang
and Zhong for moduli with at most six prime factors.  Building on their analytic estimates and optimized Selberg
sieve, we replace their truncated divisor penalty by a geometric half-weight.
The new weight retains the $P_5$ exclusion threshold and is a positive linear
combination of two standard two-parameter truncated divisor weights, so the
Zhang--Zhong transference argument applies without alteration.  After
transference, the relevant pointwise coefficient is reduced from $5/16$ to
$5/32$, which yields a positive final sieve margin.
\end{abstract}

\maketitle

\section{Introduction}

For a positive integer $q$ and an integer $m$ coprime to $q$, write
\[
  \Kl(m;q)=\frac1{\sqrt q}\sum_{\substack{x\bmod q\\(x,q)=1}}
  \e\left(\frac{mx+\overline{x}}{q}\right),
  \qquad \e(z)=e^{2\pi iz}.
\]
The horizontal distribution of $\Kl(1;q)$ as the modulus varies is a
fundamental and difficult problem.  Even the assertion that $\Kl(1;p)$ takes
both signs infinitely often over prime moduli remains open.  A breakthrough of
Fouvry and Michel \cite[Theorems~1.2--1.3]{FM}, based in part on estimates for
sums of absolute values of exponential sums developed in \cite{FM03},
established that, for each $\sigma\in\{+1,-1\}$,
\begin{equation}\label{eq:FM-result}
  \#\left\{
  X<q\le2X:
  \begin{array}{l}
    \mu^2(q)=1,\ p\mid q\Longrightarrow p>X^{1/23.9},\\
    \sigma\Kl(1;q)>0
  \end{array}
  \right\}
  \gg \frac{X}{\log X}.
\end{equation}
In particular, every modulus counted in \eqref{eq:FM-result} has at most
twenty-three distinct prime factors.  As usual, $P_r$ denotes an integer with
at most $r$ prime factors.  Subsequent work progressively lowered the number
of permitted prime factors; in particular, Matom\"aki obtained $P_{15}$
\cite{Matomaki}, while Xi ultimately obtained $P_7$ \cite{Xi1,Xi2,XiCorr}.
Drappeau and Maynard proved a conditional $P_2$ result in the presence of a
Landau--Siegel zero \cite{DM}.

Most recently, Zhang and Zhong proved unconditionally that, for each
$\sigma\in\{+1,-1\}$,
\[
  \#\{X<q\le2X:\ \mu^2(q)=1,\ \om(q)\le6,\
  \sigma\Kl(1;q)>0\}
  \gg \frac{X}{\log X}.
\]
This is the quantitative form of \cite[Theorem~1.2]{ZZv2}.  Their principal innovation is a
truncated divisor function whose support depends on the number of prime
factors of the modulus.  The purpose of the present paper is to show that a
geometric modification of this weight, while retaining the analytic inputs of
Zhang and Zhong, lowers the unconditional result from $P_6$ to $P_5$.

The main result is the following.

\begin{theorem}[Main theorem]
\label{thm:main-intro}
There exists $c_0>0$ such that, for each $\sigma\in\{+1,-1\}$ and all sufficiently large $X$,
\[
  \#\{X<n\le2X:\ \mu^2(n)=1,\ \om(n)\le5,\ \sigma\Kl(1;n)>0\}
  \ge c_0\frac X{\log X}.
\]
In particular, $\Kl(1;q)$ changes sign infinitely often as $q\to\infty$ along square-free moduli $q$ with $\om(q)\le5$.
\end{theorem}

The proof uses three precise inputs from Zhang--Zhong: the lower bound for the
sign-detecting term in their Proposition~2.1, the $R_2$ transference carried
out in their Section~4, and the optimized Selberg majorant in the proof of
their Proposition~2.2; see \cite[Propositions~2.1--2.2 and Sections~4--5]{ZZv2}.
In the normalization used below, these give a lower bound of size
$0.76235\,\wt g(1)X/\log X$, transfer the Kloosterman absolute value to
divisor weights with transformed parameters, and bound a Selberg majorant
with coefficient $5/16$ by
$6.27044\alpha_0^3\wt g(1)X/\log X$.

Our modification changes only the divisor weight used in the error term.
The geometric half-weight is arranged so that every square-free integer
with at least six prime factors is still excluded by the final sieve
inequality, while the transferred coefficient in the Zhang--Zhong
majorant is reduced from \(5/16\) to \(5/32\).  Consequently the relevant
\(R_2\)-constant is halved, from \(6.27044\alpha_0^3\) to
\(3.13522\alpha_0^3\).  This creates a nonempty admissible window for the
final cutoff \(\rho=c\alpha_0^3\); for example \(c=19/2\) lies in this
window.  The final sieve sum is therefore positive, while all terms with
\(\omega(n)\ge6\) remain excluded, forcing the positive contribution to
come from \(P_5\) moduli.

The rest of the paper is organized as follows.
Section~\ref{sec:zz-inputs} fixes notation and records the precise
Zhang--Zhong estimates used below.  Section~\ref{sec:weight} introduces
the geometric half-weight and proves the \(P_5\) exclusion threshold.
Section~\ref{sec:R2} proves the transferred coefficient bound and the
improved \(R_2\) estimate.  Section~\ref{sec:sieve} closes the sieve.

\section{Notation and Zhang--Zhong estimates}\label{sec:zz-inputs}

Let $g$ be a nonnegative smooth function supported in $[1,2]$, and define its Mellin transform by
\[
  \wt g(s)=\int_0^\infty g(x)x^{s-1}\,dx,
  \qquad \wt g(1)>0.
\]
Let
\[
  \Pi_\eps=\prod_{p<X^\eps}p,
\]
where $\eps>0$ is fixed and sufficiently small.  We use the Selberg sieve weights
\[
  \lambda_d=\mu(d)F\left(\frac{\log(\sqrt D/d)}{\log\sqrt D}\right),
  \qquad D=X^{1/2-\eps},
\]
with the convention \(\lambda_d=0\) for \(d>\sqrt D\).  On \([0,1]\) we take
\[
  F(x)=x^4(a_0+a_1x+a_2x^2+a_3x^3+a_4x^4),
\]
and extend \(F\) by zero outside \([0,1]\).  The polynomial is normalized as in
Zhang--Zhong so that
\[
  F(1)=1,\qquad |F(x)|\le1\quad(0\le x\le1).
\]
Consequently
\[
  \lambda_1=1,\qquad |\lambda_d|\le1\quad(d\ge1).
\]
The optimized coefficients are approximately
\[
  (a_0,a_1,a_2,a_3,a_4)
  \approx\left(
  \frac{100396}{53901},
  -\frac{17284}{13475},
  \frac{76486}{134753},
  -\frac{33241}{188654},
  \frac{10836}{377308}
  \right),
\]
following \cite[Section~2.2 and p.~13]{ZZv2}.

For square-free $n$, define
\begin{equation*}
  \calB(n)=\{d\mid n:\ \om(d)=3,\ \om(n/d)\ge3,\ d\le n^{1/2}\}.
\end{equation*}
For fixed $\alpha,\beta>0$, put
\begin{equation}
\label{eq:tau-alpha-beta}
  \tau(n;\alpha,\beta)=\sum_{d\in\calB(n)}\alpha^{\om(d)}\beta^{\om(n/d)}.
\end{equation}
Thus each $d\in\calB(n)$ contributes $\alpha^3\beta^{\om(n)-3}$.  We also write
\[
  N(n)=\#\calB(n).
\]

Following Zhang--Zhong \cite[Equation (9)]{ZZv2}, after fixing the Selberg polynomial $F$ above, we write
$\alpha_0$ for the fixed value of the first divisor-weight parameter used in
their numerical $R_2$ computation with $\beta=1$.  We keep this normalization
and do not re-optimize $\alpha_0$.  Since every $d\in\calB(n)$ has
$\om(d)=3$, all constants involving this parameter are measured on the common
scale $\alpha_0^3$.  In particular, the constants $6.27044\alpha_0^3$ and
$3.13522\alpha_0^3$, the exclusion threshold $10\alpha_0^3$, and the cutoff
$\rho$ are compared only through the ratio $\rho/\alpha_0^3$.

We record three results from the Zhang--Zhong framework in the exact form needed here.  The second proposition is written with an arbitrary fixed value of $\beta$ because the proof in \cite[Section~4, especially equations~(6)--(8)]{ZZv2} carries $\beta$ as a fixed positive parameter until the final numerical specialization $\beta=1$ in their equation~(9).

\begin{proposition}[Zhang--Zhong lower bound for the sign-detecting term]\label{prop:ZZ-R1}
For each $\sigma\in\{+1,-1\}$, define
\[
  R_1^\sigma(X)=\sum_{(n,\Pi_\eps)=1}g\left(\frac nX\right)\mu^2(n)
  \bigl(\abs{\Kl(1;n)}+\sigma\Kl(1;n)\bigr)
  \left(\sum_{d\mid n}\lambda_d\right)^2.
\]
Then
\[
  R_1^\sigma(X)
  \ge (1+o(1))\,0.76235\,\wt g(1)\frac X{\log X}.
\]
\end{proposition}

\begin{proof}
This is \cite[Proposition~2.1]{ZZv2}; see also the proof in \cite[Section~5]{ZZv2}, where the contribution from integers with small prime factors is shown to be negligible after $\eps$ is chosen sufficiently small.
\end{proof}

\begin{proposition}[Zhang--Zhong \(R_2\)-transference]
\label{prop:ZZ-transfer}
For fixed \(\alpha,\beta>0\), define
\begin{equation}
\label{eq:R2-alpha-beta}
  R_2(X;\alpha,\beta)
  =
  \sum_{(n,\Pi_\eps)=1}
  g\left(\frac nX\right)\mu^2(n)\abs{\Kl(1;n)}
  \tau(n;\alpha,\beta)
  \left(\sum_{r\mid n}\lambda_r\right)^2 .
\end{equation}
Then
\begin{equation}
\label{eq:ZZ-transfer}
\begin{split}
  R_2(X;\alpha,\beta)
  &\le
  (1+o(1))
  \sum_n
  g\left(\frac nX\right)\mu^2(n)
  \tau\left(n;\frac{8\alpha}{3\pi},2\beta\right)
  \left(\sum_{r\mid n}\lambda_r\right)^2 \\
  &\quad +o\left(\frac X{\log X}\right).
\end{split}
\end{equation}
Here, \(\alpha\) and \(\beta\) are fixed while \(X\to\infty\); in particular, this estimate applies to the finitely many fixed parameter pairs used below.
\end{proposition}

\begin{proof}
This follows Zhang and Zhong's argument in \cite[Section~4, equations~(6)--(8)]{ZZv2}. We briefly recall the reduction.

We split the truncated divisor function according to whether the selected divisor is separated from the square-root boundary:
\[
  \tau(n;\alpha,\beta)=\tau_{<}(n;\alpha,\beta)+\tau_{>}(n;\alpha,\beta),
\]
where
\[
  \tau_{<}(n;\alpha,\beta)
  =
  \sum_{\substack{m\in\calB(n)\\
  m\le n^{1/2}\exp(-\sqrt{\log X})}}
  \alpha^{\om(m)}\beta^{\om(n/m)}
\]
and
\[
  \tau_{>}(n;\alpha,\beta)
  =
  \sum_{\substack{m\in\calB(n)\\
  n^{1/2}\exp(-\sqrt{\log X})<m\le n^{1/2}}}
  \alpha^{\om(m)}\beta^{\om(n/m)}.
\]
Let \(R_{21}(X;\alpha,\beta)\) and \(R_{22}(X;\alpha,\beta)\) denote the corresponding contributions to \(R_2(X;\alpha,\beta)\). Thus, following the notation in \cite[Section~4]{ZZv2}, we have
\begin{equation}\label{eq:R2-split-ZZ}
  R_2(X;\alpha,\beta)
  =
  R_{21}(X;\alpha,\beta)+R_{22}(X;\alpha,\beta).
\end{equation}

As shown in \cite[Section~4, equation~(6)]{ZZv2}, the term \(R_{22}\) is negligible:
\begin{equation}\label{eq:R22-negligible}
  R_{22}(X;\alpha,\beta)
  \ll_{\alpha,\beta} X(\log X)^{-5/4}
  =
  o\left(\frac X{\log X}\right).
\end{equation}

Furthermore, the estimate in
\cite[Section~4, p.~13]{ZZv2} gives
\begin{equation}\label{eq:R21-transfer}
  R_{21}(X;\alpha,\beta)
  \le
  (1+o(1))
  \sum_n
  g\left(\frac nX\right)\mu^2(n)
  \tau\left(n;\frac{8\alpha}{3\pi},2\beta\right)
  \left(\sum_{r\mid n}\lambda_r\right)^2 .
\end{equation}

Finally, combining \eqref{eq:R2-split-ZZ}, \eqref{eq:R22-negligible}, and \eqref{eq:R21-transfer} yields \eqref{eq:ZZ-transfer}.
\end{proof}

\begin{proposition}[Zhang--Zhong Selberg main term]\label{prop:ZZ-Selberg}
For the polynomial \(F\) above,
\[
\begin{aligned}
  &\frac5{16}\left(\frac{4\alpha_0}{3\pi}\right)^3
  \sum_{\om(n)\ge6}
  g\left(\frac nX\right)\mu^2(n)4^{\om(n)}
  \left(\sum_{d\mid n}\lambda_d\right)^2  \\
  &\qquad\le
  (1+o(1))\,6.27044\,\alpha_0^3\,\wt g(1)\frac X{\log X}.
\end{aligned}
\]
Equivalently, replacing the coefficient \(5/16\) by \(5/32\) replaces the
constant \(6.27044\) by \(3.13522\).
\end{proposition}

\begin{proof}
This is the estimate obtained from equations~(8)--(9) in Zhang--Zhong's proof
of Proposition~2.2; see \cite[pp.~13--14]{ZZv2}.
\end{proof}

\section{The geometric half-weight and the \texorpdfstring{$P_5$}{P5} threshold}\label{sec:weight}

We first record the elementary combinatorial input on the divisor set
\(\calB(n)\).  Recall that
\[
  N(n)=\#\calB(n).
\]

\begin{lemma}[Zhang--Zhong combinatorial lemma]
\label{lem:comb}
Let \(n\) be square-free. Then
\[
  N(n)=10 \qquad(\om(n)=6),
\]
and
\[
  N(n)\ge20 \qquad(\om(n)\ge7).
\]
\end{lemma}

\begin{proof}
The cases \(\om(n)=6\) and \(\om(n)=7\) are exactly
\cite[Lemma~A.1]{ZZv2}.  If \(\om(n)\ge8\), choose seven prime factors of
\(n\) and let \(m\) be their product.  Applying the \(\om=7\) case to \(m\)
gives at least \(20\) divisors \(d\mid m\) with \(\om(d)=3\) and
\(d\le m^{1/2}\).  For each such \(d\), we have \(d\le m^{1/2}\le n^{1/2}\)
and
\[
  \om(n/d)=\om(n)-3\ge5.
\]
Thus \(d\in\calB(n)\), and hence \(N(n)\ge20\).
\end{proof}

We now introduce the auxiliary factor used in the final sieve:
\[
  \rho-\tau_{5,t}(n;\alpha_0).
\]
Here \(\rho>0\) will be chosen later, and \(\tau_{5,t}(n;\alpha_0)\) denotes
the geometric \(P_5\) penalty weight defined below.
To force all non-\(P_5\) moduli to contribute non-positively, we need a
pointwise lower bound
\[
  \tau_{5,t}(n;\alpha_0)\ge 10\alpha_0^3
  \qquad(\om(n)\ge6),
\]
and then we will choose \(\rho<10\alpha_0^3\).  At the same time, the new
weight must be compatible with Zhang--Zhong's transference estimate, which is
available for the two-parameter weights \(\tau(n;\alpha,\beta)\) defined in
\eqref{eq:tau-alpha-beta}.

Lemma~\ref{lem:comb} suggests the following ideal divisor profile.  An
admissible divisor \(d\in\calB(n)\) should receive full weight when
\(\om(n/d)=3\), and half weight when \(\om(n/d)\ge4\).  Indeed, if
\(\om(n)=6\), then every \(d\in\calB(n)\) has \(\om(n/d)=3\), and
Lemma~\ref{lem:comb} gives exactly \(10\) admissible divisors.  Full weight on
each of them gives the threshold \(10\alpha^3\).  If \(\om(n)\ge7\), then every
\(d\in\calB(n)\) has \(\om(n/d)\ge4\), and Lemma~\ref{lem:comb} gives at least
\(20\) admissible divisors.  Half weight per divisor is therefore already enough
to give the same threshold \(10\alpha^3\).

This discontinuous ideal profile is not directly expressible in terms of the
Zhang--Zhong two-parameter weights.  We therefore use a geometric
substitute.  Fix
\[
  0<t\le\frac17;
\]
the concrete choice \(t=1/8\) is convenient.  For integers \(s\ge3\), define
\[
  h_t(s)=\frac12+\frac12t^{s-3}.
\]
Then
\[
  h_t(3)=1,
  \qquad
  h_t(s)\ge\frac12 \quad(s\ge4).
\]
The corresponding \(P_5\) penalty weight is
\begin{equation}\label{eq:tau5-def}
  \tau_{5,t}(n;\alpha)
  =
  \alpha^3\sum_{d\in\calB(n)}h_t(\om(n/d)).
\end{equation}

The reason for using the geometric profile is the following exact identity:
\begin{equation}
\label{eq:linear-combination}
  \tau_{5,t}(n;\alpha)
  =
  \frac12\tau(n;\alpha,1)
  +
  \frac{1}{2t^3}\tau(n;\alpha,t).
\end{equation}
Indeed, since every \(d\in\calB(n)\) has \(\om(d)=3\), the right-hand side is
\[
  \alpha^3
  \sum_{d\in\calB(n)}
  \left(\frac12+\frac12t^{\om(n/d)-3}\right),
\]
which is exactly \eqref{eq:tau5-def}.  Thus the new penalty is a positive
linear combination of two standard Zhang--Zhong weights, so the transference
Proposition \ref{prop:ZZ-transfer} can be applied separately with \(\beta=1\) and \(\beta=t\).
The restriction \(t\le1/7\) will be used later, in Lemma~\ref{lem:five32}, to
obtain the transferred coefficient \(5/32\).

\begin{lemma}[The \(P_5\) threshold]\label{lem:threshold}
For every square-free integer \(n\), we have
\[
  \begin{cases}
    \tau_{5,t}(n;\alpha) \ge 10\alpha^3 & \text{if } \om(n) \ge 6, \\
    \tau_{5,t}(n;\alpha) = 0            & \text{if } \om(n) \le 5.
  \end{cases}
\]
\end{lemma}

\begin{proof}
If \(\om(n)\le5\), then \(\calB(n)=\varnothing\).  Hence \(\tau_{5,t}(n;\alpha)=0\).

Now suppose \(\om(n)\ge6\).  If \(\om(n)=6\), then \(\om(n/d)=3\) for every
\(d\in\calB(n)\).  Since \(h_t(3)=1\), Lemma~\ref{lem:comb} gives
\[
  \tau_{5,t}(n;\alpha)
  =
  \alpha^3N(n)
  =
  10\alpha^3.
\]
If \(\om(n)\ge7\), then \(\om(n/d)\ge4\) for every \(d\in\calB(n)\).
Consequently, \(h_t(\om(n/d))\ge1/2\).  Again by Lemma~\ref{lem:comb},
\[
  \tau_{5,t}(n;\alpha)
  \ge
  \frac12\alpha^3N(n)
  \ge
  10\alpha^3.
\]
This proves the threshold.
\end{proof}

\section{The transferred coefficient and the new \texorpdfstring{$R_2$}{R2} estimate}\label{sec:R2}

Define
\begin{equation*}
  R_{2,t}(X)
  =
  \sum_{(n,\Pi_\eps)=1}
  g\left(\frac nX\right)\mu^2(n)\abs{\Kl(1;n)}
  \tau_{5,t}(n;\alpha_0)
  \left(\sum_{d\mid n}\lambda_d\right)^2,
\end{equation*}
Here we use a different \(P_5\) penalty weight from \eqref{eq:R2-alpha-beta}.
The goal of this section is to show that the geometric half-weight reduces by
a factor of two the Zhang--Zhong $R_2$ constant in
\cite[Proposition~2.2]{ZZv2}.

To estimate \(R_{2,t}(X)\), we do not introduce a new sieve weight.  We keep
\(\tau_{5,t}(n;\alpha_0)\) in the final sieve, but use Zhang--Zhong's
transference proposition to majorize the \(R_2\)-sum.  Since
\[
  \tau_{5,t}(n;\alpha)
  =
  \frac12\tau(n;\alpha,1)
  +
  \frac{1}{2t^3}\tau(n;\alpha,t),
\]
Proposition~\ref{prop:ZZ-transfer} may be applied separately to the two
positive terms.  At the level of the resulting Selberg majorant, each standard
weight undergoes the replacement
\[
  \tau(n;\alpha,\beta)
  \longmapsto
  \tau\left(n;\frac{8\alpha}{3\pi},2\beta\right).
\]
Thus the transferred coefficient which has to be bounded pointwise is
\[
  \frac12
  \tau\left(n;\frac{8\alpha}{3\pi},2\right)
  +
  \frac{1}{2t^3}
  \tau\left(n;\frac{8\alpha}{3\pi},2t\right).
\]

\begin{lemma}[Transferred coefficient bound]
\label{lem:five32}
Let \(0<t\le1/7\) and \(\alpha>0\).  For every square-free integer \(n\),
\begin{equation}
\label{eq:transferred-pointwise-bound}
\begin{aligned}
  &\frac12
  \tau\left(n;\frac{8\alpha}{3\pi},2\right)
  +
  \frac{1}{2t^3}
  \tau\left(n;\frac{8\alpha}{3\pi},2t\right)  \\
  &\qquad\le
  \frac5{32}
  \left(\frac{4\alpha}{3\pi}\right)^3
  4^{\om(n)}.
\end{aligned}
\end{equation}
\end{lemma}

\begin{proof}
If \(\calB(n)=\varnothing\), the left-hand side is zero and there is nothing
to prove.  We may therefore assume \(\calB(n)\ne\varnothing\).  Then
\(\om(n)\ge6\).  Put
\[
  r=\om(n).
\]
For every \(d\in\calB(n)\), we have \(\om(d)=3\), and hence
\(\om(n/d)=r-3\).  Therefore the left-hand side of
\eqref{eq:transferred-pointwise-bound} is
\[
\begin{aligned}
  &\sum_{d\in\calB(n)}
  \left[
  \frac12
  \left(\frac{8\alpha}{3\pi}\right)^3
  2^{r-3}
  +
  \frac{1}{2t^3}
  \left(\frac{8\alpha}{3\pi}\right)^3
  (2t)^{r-3}
  \right]  \\
  &\qquad=
  \left(\frac{4\alpha}{3\pi}\right)^3
  4^r
  N(n)\frac{1+t^{r-6}}{2^{r+1}},
\end{aligned}
\]
where \(N(n)=\#\calB(n)\).

It remains to prove
\[
  N(n)\frac{1+t^{r-6}}{2^{r+1}}\le\frac5{32}.
\]
If \(r=6\), Lemma~\ref{lem:comb} gives \(N(n)=10\), and hence
\[
  N(n)\frac{1+t^{r-6}}{2^{r+1}}
  =
  10\frac{2}{2^7}
  =
  \frac5{32}.
\]
If \(r\ge7\), we use the trivial bound \(N(n)\le\binom r3\).  The sequence
\[
  \frac{\binom r3}{2^{r+1}}
\]
is non-increasing for \(r\ge5\), and \(1+t^{r-6}\) is non-increasing for
\(r\ge7\) when \(0<t\le1\).  Thus the maximum for \(r\ge7\) occurs at \(r=7\).
Since $t\le1/7$,
\[
  N(n)\frac{1+t^{r-6}}{2^{r+1}}
  \le
  \binom73\frac{1+t}{2^8}
  \le
  35\frac{1+1/7}{256}
  =
  \frac5{32}.
\]
This proves the lemma.
\end{proof}

\begin{proposition}[New \(R_2\) bound]\label{prop:newR2}
For every fixed \(0<t\le1/7\),
\[
  R_{2,t}(X)
  \le
  (1+o(1))\,3.13522\,\alpha_0^3\,\wt g(1)\frac X{\log X}.
\]
\end{proposition}

\begin{proof}
By the positive decomposition \eqref{eq:linear-combination},
\[
  \tau_{5,t}(n;\alpha_0)
  =
  \frac12\tau(n;\alpha_0,1)
  +
  \frac{1}{2t^3}\tau(n;\alpha_0,t).
\]
We may apply Proposition~\ref{prop:ZZ-transfer} separately to the two terms.
Since \(t\) is fixed, the two error terms are still \(o(X/\log X)\).  Put
\[
  T_t(n)=
  \frac12
  \tau\left(n;\frac{8\alpha_0}{3\pi},2\right)
  +
  \frac{1}{2t^3}
  \tau\left(n;\frac{8\alpha_0}{3\pi},2t\right).
\]
The transference estimate gives
\[
\begin{split}
  R_{2,t}(X)
  &\le (1+o(1))
  \sum_n g\left(\frac nX\right)\mu^2(n)T_t(n)
  \left(\sum_{d\mid n}\lambda_d\right)^2 \\
  &\quad+o\left(\frac X{\log X}\right).
\end{split}
\]
Here the restriction \((n,\Pi_\eps)=1\) has been dropped in the majorant, since
all terms are non-negative.

Applying Lemma~\ref{lem:five32} with \(\alpha=\alpha_0\), we obtain
\[
\begin{aligned}
  R_{2,t}(X)
  &\le
  (1+o(1))\frac5{32}
  \left(\frac{4\alpha_0}{3\pi}\right)^3
  \sum_{\om(n)\ge6}
  g\left(\frac nX\right)\mu^2(n)4^{\om(n)}
  \left(\sum_{d\mid n}\lambda_d\right)^2   \\
  &\qquad+
  o\left(\frac X{\log X}\right).
\end{aligned}
\]
Proposition~\ref{prop:ZZ-Selberg}, with \(5/16\) replaced by \(5/32\), gives
\[
  R_{2,t}(X)
  \le
  (1+o(1))\,3.13522\,\alpha_0^3\,\wt g(1)\frac X{\log X}.
\]
This proves the proposition.
\end{proof}

\section{Closing the sieve}\label{sec:sieve}

Fix, for instance, $t=1/8$.

For \(\sigma\in\{+1,-1\}\) and \(\rho>0\), define the signed sieve sum
\begin{equation*}
\begin{aligned}
  \mathcal S_{t,\rho}^{\sigma}(X)
  =
  \sum_{(n,\Pi_\eps)=1}
  &g\left(\frac nX\right)\mu^2(n)
  \bigl(\abs{\Kl(1;n)}+\sigma\Kl(1;n)\bigr)  \\
  &\times
  \bigl(\rho-\tau_{5,t}(n;\alpha_0)\bigr)
  \left(\sum_{d\mid n}\lambda_d\right)^2 .
\end{aligned}
\end{equation*}
The factor \(\abs{\Kl(1;n)}+\sigma\Kl(1;n)\) detects the sign \(\sigma\):
\[
  \abs{\Kl(1;n)} + \sigma\Kl(1;n)
  =
  \begin{cases}
    2\abs{\Kl(1;n)} & \text{if } \sigma\Kl(1;n) > 0, \\
    0               & \text{if } \sigma\Kl(1;n) \le 0.
  \end{cases}
\]

We first prove positivity of \(\mathcal S_{t,\rho}^{\sigma}(X)\) for a suitable
choice of \(\rho\).  Expanding only the factor
\(\rho-\tau_{5,t}(n;\alpha_0)\), we obtain
\begin{multline*}
  \mathcal S_{t,\rho}^{\sigma}(X)
  =
  \rho R_1^\sigma(X)  \\
  -
  \sum_{(n,\Pi_\eps)=1}
  g\left(\frac nX\right)\mu^2(n)
  \bigl(\abs{\Kl(1;n)}+\sigma\Kl(1;n)\bigr)
  \tau_{5,t}(n;\alpha_0)
  \left(\sum_{d\mid n}\lambda_d\right)^2 .
\end{multline*}
Since \(g\ge0\), \(\tau_{5,t}(n;\alpha_0)\ge0\), and
\[
  0\le \abs{\Kl(1;n)}+\sigma\Kl(1;n)\le2\abs{\Kl(1;n)},
\]
the second term is at most \(2R_{2,t}(X)\).  Hence
\begin{equation}
\label{eq:sieve-lower-bound}
  \mathcal S_{t,\rho}^{\sigma}(X)
  \ge
  \rho R_1^\sigma(X)-2R_{2,t}(X).
\end{equation}

The choice of $\rho$ is constrained from both sides.  Write
$\rho=c\alpha_0^3$.  By Proposition~\ref{prop:ZZ-R1},
Proposition~\ref{prop:newR2}, and \eqref{eq:sieve-lower-bound}, positivity is
obtained as soon as
\[
  0.76235c-2\cdot3.13522>0.
\]
Equivalently,
\[
  c>\frac{2\cdot3.13522}{0.76235}=8.22515\ldots .
\]
On the other hand, Lemma~\ref{lem:threshold} excludes every modulus with
$\om(n)\ge6$ only if
\[
  \rho<10\alpha_0^3,
  \qquad\text{that is,}\qquad c<10.
\]
Thus any fixed $c$ in the open interval
\[
  8.22515\ldots<c<10
\]
would be admissible.  We choose the convenient value
\[
  \rho=\frac{19}{2}\alpha_0^3,
\]
which leaves a margin below the exclusion threshold and also gives a positive
main-term margin.  Then Proposition~\ref{prop:ZZ-R1} and
Proposition~\ref{prop:newR2} give
\[
\begin{aligned}
  \mathcal S_{t,\rho}^{\sigma}(X)
  &\ge
  (1+o(1))
  \left(
  0.76235\cdot\frac{19}{2}
  -
  2\cdot3.13522
  \right)
  \alpha_0^3\wt g(1)\frac X{\log X}  \\
  &=
  (1+o(1))\,0.971885\,\alpha_0^3\wt g(1)\frac X{\log X}.
\end{aligned}
\]
In particular,
\begin{equation}\label{eq:sieve-positive}
  \mathcal S_{t,\rho}^{\sigma}(X)\gg \frac X{\log X}
\end{equation}
for \(X\) sufficiently large.

It remains to identify where this positive contribution can come from.
By Lemma~\ref{lem:threshold},
\[
  \tau_{5,t}(n;\alpha_0)\ge10\alpha_0^3
  >
  \frac{19}{2}\alpha_0^3
  =
  \rho
  \qquad(\om(n)\ge6).
\]
Therefore every term with \(\om(n)\ge6\) contributes non-positively to
\(\mathcal S_{t,\rho}^{\sigma}(X)\).  On the other hand, if \(\om(n)\le5\),
then \(\calB(n)=\varnothing\), and hence
\[
  \tau_{5,t}(n;\alpha_0)=0.
\]
Thus the positive part of \(\mathcal S_{t,\rho}^{\sigma}(X)\) is supported only
on square-free \(P_5\) moduli with sign \(\sigma\Kl(1;n)>0\).

Let
\[
  \calP_\sigma(X)
  =
  \left\{n:
  \begin{array}{l}
    g(n/X)\ne0,\ (n,\Pi_\eps)=1,\ \mu^2(n)=1,\\
    \om(n)\le5,\ \sigma\Kl(1;n)>0
  \end{array}
  \right\}.
\]
For \(n\in\calP_\sigma(X)\), the Weil bound for Kloosterman sums, in the form
\cite[Corollary~11.12]{IwKow04}, gives
\[
  \abs{\Kl(1;n)}
  \le 
  2^{\om(n)}
  \le32.
\]
Moreover, by the construction of the Selberg weights, \(\abs{\lambda_d}\le1\).
Hence
\[
  \left|\sum_{d\mid n}\lambda_d\right|
  \le
  \sum_{d\mid n}1
  =
  2^{\om(n)}
  \le32.
\]
Since \(g\) and \(\rho\) are fixed, each positive summand with
\(\om(n)\le5\) is \(O_{g,\rho}(1)\).  All remaining summands are
non-positive.  Hence
\[
  0<
  \mathcal S_{t,\rho}^{\sigma}(X)
  \ll_{g,\rho}
  \#\calP_\sigma(X).
\]
Together with \eqref{eq:sieve-positive}, this gives
\[
  \#\calP_\sigma(X)\gg_{g,\rho}\frac X{\log X}.
\]
Taking \(\sigma=+1\) and \(\sigma=-1\), respectively, gives positive and
negative Kloosterman values on square-free \(P_5\) moduli in every sufficiently
large dyadic interval.  This proves Theorem~\ref{thm:main-intro}.

\section*{Acknowledgements}
The author would like to thank Igor E. Shparlinski for his helpful discussions and for reading an earlier draft of this manuscript.
This work was supported by the China Scholarship Council.

\end{document}